\documentclass[a4paper,10pt,reqno]{amsart}

\usepackage[utf8]{inputenc}
\usepackage[T1]{fontenc}
\usepackage{amsthm}
\usepackage{amsmath}
\usepackage{amssymb}
\usepackage[inline, shortlabels]{enumitem}
\usepackage{comment}
\PassOptionsToPackage{hyphens}{url}
\usepackage{hyperref}
\usepackage{fancyhdr}
\usepackage{mathrsfs}
\usepackage{stmaryrd}
\usepackage[normalem]{ulem}
\usepackage{xcolor}
\usepackage[skip=0pt]{caption}
\usepackage{tikz-cd}

\theoremstyle{definition}

\newtheorem{theorem}{Theorem}[section]

\newtheorem{proposition}[theorem]{Proposition}

\theoremstyle{definition}

\newtheorem{remark}[theorem]{Remark}

\newtheorem{claim}{Claim}

\definecolor{blue-url}{RGB}{0,0,100}
\definecolor{red-url}{RGB}{100,0,0}
\definecolor{green-url}{RGB}{0,100,0}
\definecolor{light-yellow}{RGB}{255,255,128}
\definecolor{light-blue}{RGB}{193,255,255}
\definecolor{light-red}{RGB}{239,83,80}

\hypersetup{
	pdftitle={On the density of sumsets, II},
	pdfauthor={Leonetti and Tringali},
	pdfmenubar=false,
	pdffitwindow=true,
	pdfstartview=FitH,
	colorlinks=true,
	linkcolor=blue-url,
	citecolor=green-url,
	urlcolor=red-url
}

\renewcommand{\emptyset}{\varnothing}
\renewcommand{\ast}{\star}
\renewcommand{\,}{\kern 0.1em}

\providecommand\llb{\llbracket}
\providecommand\rrb{\rrbracket}

\newcommand{\evid}[1]{\textsf{#1}}

{\newline\vspace{\abovedisplayskip}\hbox to \textwidth\bgroup\hss$\displaystyle}
{$\hss\egroup\vspace{\belowdisplayskip}}

\makeatletter
\DeclareFontFamily{OMX}{MnSymbolE}{}
\DeclareSymbolFont{MnLargeSymbols}{OMX}{MnSymbolE}{m}{n}
\SetSymbolFont{MnLargeSymbols}{bold}{OMX}{MnSymbolE}{b}{n}
\DeclareFontShape{OMX}{MnSymbolE}{m}{n}{
	<-6>  MnSymbolE5
	<6-7>  MnSymbolE6
	<7-8>  MnSymbolE7
	<8-9>  MnSymbolE8
	<9-10> MnSymbolE9
	<10-12> MnSymbolE10
	<12->   MnSymbolE12
}{}
\DeclareFontShape{OMX}{MnSymbolE}{b}{n}{
	<-6>  MnSymbolE-Bold5
	<6-7>  MnSymbolE-Bold6
	<7-8>  MnSymbolE-Bold7
	<8-9>  MnSymbolE-Bold8
	<9-10> MnSymbolE-Bold9
	<10-12> MnSymbolE-Bold10
	<12->   MnSymbolE-Bold12
}{}

\AtBeginDocument{%
   \def\MR#1{}
}

\let\llangle\@undefined
\let\rrangle\@undefined
\DeclareMathDelimiter{\llangle}{\mathopen}%
{MnLargeSymbols}{'164}{MnLargeSymbols}{'164}
\DeclareMathDelimiter{\rrangle}{\mathclose}%
{MnLargeSymbols}{'171}{MnLargeSymbols}{'171}
\makeatother
\begin{document}
\title{On the density of sumsets, II}
\author{Paolo Leonetti}
\address[P.~Leonetti]{Department of Economics, Universit\`a degli Studi dell'Insubria, via Monte Generoso 71 | Varese 21100, Italy}
\email{leonetti.paolo@gmail.com}
\urladdr{https://sites.google.com/site/leonettipaolo}
\author[S.~Tringali]{Salvatore Tringali}
\address[S.~Tringali]{School of Mathematical Sciences,
Hebei Normal University | Shijiazhuang, Hebei province, 050024 China}
\email{salvo.tringali@gmail.com}
\urladdr{http://imsc.uni-graz.at/tringali}
%
%\thanks{P.L. was supported by the Austrian Science Fund (FWF), project F5512-N26 and by PRIN 2017, grant 2017CY2NCA.}
%
% 11B05: Density, gaps, topology
% 11B13: Additive bases, including sumsets
% 28A10: Real- or complex-valued set functions
% 39B52: Equations for functions with more general domains and/or ranges
% 39B62: Functional inequalities, including subadditivity, convexity, etc.
% 60B99: None of the above, but in this section (Probability measures)
%
\begin{abstract}
\noindent{}
Arithmetic quasi-densities are a large family of real-valued set functions partially defined on the power set of $\mathbb{N}$, including the asymptotic density, the Banach density, the analytic density, etc.

Let $B \subseteq \mathbb{N}$ be a non-empty set covering $o(n!)$ residue classes modulo $n!$ as $n\to \infty$ (e.g., the primes or the perfect powers). We show that, for each $\alpha \in [0,1]$, there is a set $A\subseteq \mathbb{N}$ such that, for every arithmetic quasi-density $\mu$, both $A$ and the sumset $A+B$ are in the domain of $\mu$ and, in addition, $\mu(A + \allowbreak B) = \alpha$. The proof relies on the properties of a little known density first considered by Buck in 1946.
\end{abstract}
\maketitle
\thispagestyle{empty}

%%%%%%%%%%%%%%%%%%%%%%%%%%%%%%%%%%%%%%%%%
%%%%%%%%%%%%%%%%%%%%%%%%%%%%%%%%%%%%%%%%%%

\section{Introduction}
\label{sec:intro}

Let $\mathsf d$ be the \evid{asymptotic} (or \evid{natural}) \evid{density} on the non-negative integers $\mathbb N$ and $\mathrm{dom}(\mathsf{d})$ be the family of all sets $X \subseteq \mathbb{N}$ which possess asymptotic density, meaning that the limit of $\frac{1}{n} \bigl| X \cap [1, n] \bigr|$ as $n \to \infty$ exists. 
We are going to show that, if $B \subseteq \mathbb N$ is non-empty and ``sufficiently small'', then
there is a family of sets of the form $A+B$ with $A$ and $A+B$ both in $\mathrm{dom}(\mathsf d)$ such that the corresponding asymptotic densities attain every value in the interval $[0,1]$, where 
$$
A + B := \{x + y \colon x \in A,\, y \in B\}
$$
is the \evid{sumset} of $A$ and $B$ (see Sect.~\ref{sec:densities} for details and examples). Writing $\mathbb{P}:=\{2,3,5,\ldots\}$ for the set of primes, we obtain as a special case the following:

\begin{theorem}\label{thm:primes}
%Let $\mathbb{P}$ be the set of primes. Then, f
For each $\alpha \in [0,1]$, there exists $A \in \mathrm{dom}(\mathsf{d})$ such that $A+\mathbb{P} \in \mathrm{dom}(\mathsf{d})$ and $\mathsf{d}(A+\mathbb{P})=\alpha$. 
\end{theorem}

In fact, our main result (Theorem \ref{thm:mainsumsets2}) is much more general and stronger. Not only does it allow us to show that Theorem \ref{thm:primes} holds with the primes replaced by a greater variety of sets and the asymptotic density $\mathsf{d}$ replaced by any in a large class of axiomatically defined ``densities'' $\mu$ (including, among others, the Banach density, the analytic density, and the logarithmic density), but also it does so \emph{uniformly} in the choice of $\mu$. 
%(We stated Theorem \ref{thm:primes} mainly for the sake of exposition.)  
The result, whose proof relies on the structural properties of a little known density first considered by Buck \cite{MR18196}, belongs to a vast literature on the interplay between the structure of sumsets and their ``largeness'', see, e.g., \cite{MR4080054, MR4197427, MR4474053, MR3939472, 
MR4439850, MR3919363, MR90618} and \cite{MR4452446, MR2956240, MR4494594, MR2522038}. %, MR201410}.
%Leo22PAMS, PROCEEDINGS
%MR4332837 BULLETIN LONDON

An analogue of Theorem \ref{thm:primes} with $\mathbb{P}$ replaced by a non-empty \emph{finite} set $B \subseteq \mathbb N$ (and without the additional requirement that $A \in \mathrm{dom}(\mathsf{d})$) was proved by Faisant et al.~in \cite[Theorem 2.2]{MR4197427}. 
A previous attempt to extend the latter to an \emph{infinite} set $B$ was made by Chu in \cite[Theorem 1.5]{MR4414919}. 
However, it turned out that the proof is flawed \cite{MR4474049}. Therefore, Theorem \ref{thm:primes} provides the first example of an infinite set satisfying the required claim. (The proof is based on completely different ideas from \cite{MR4414919, MR4197427}.)
%see also \cite[Theorem 1.2]{MR4439850} for a generalization

%Secondly, the proof of our main result cannot rely on probabilistic arguments. 
%Indeed, it is known that '' ``most'' subsets $X\subseteq \mathbb{N}$ cannot be written as non-trivial sumsets $A+B$, from both measure and category viewpoint, see \cite{Leo22PAMS, MR58655}. 
%Worsening the situation, '' ``most'' subsets $A\subseteq \mathbb{N}$ have upper asymptotic density $\ge 1/2$, again from both measure and category viewpoint: this implies that, for each given non-empty set $B\subseteq \mathbb{N}$, the sumset $A+B$ contains a translation of $A$, hence in most cases the upper asymptotic density cannot be smaller than $1/2$. 

%Lastly, as already anticipated, our result holds uniformly with respect to a large family of  ``densities'' introduced axiomatically in \cite{MR4054777}. 
%, see Definition \ref{def:densities} below. 
%Therefore, the construction of the claimed set $A$ 
%in Theorem \ref{thm:primes} 
%does not rely on the definition on the asymptotic density $\mathsf{d}$, but on a short list of properties shared with the all the other functions in such family. 

\subsection*{Notation} 
We use $\mathbb Z$ for the integers and $\mathbb N^+$ for the positive integers.
Given $X \subseteq \mathbb{N}$ and $q \in \mathbb{N}$, we define $q \cdot X := \{qx \colon x \in X\}$ and $X + q := X + \{q\}$. We let an \evid{arithmetic progression} (AP) be a set of the form $k\cdot \mathbb{N}+h$ with $k \in \mathbb N^+$ and $h \in \mathbb N$, and we denote by $\mathscr{A}$ the family of all finite unions of APs. 
Finally, for each $a, b \in \mathbb{Z}$ we write $\llb a,b \rrb := [a,b] \cap \mathbb{Z}$ for the \evid{discrete interval} between $a$ and $b$.

\section{Preliminaries and Main Result}
\label{sec:densities}

We say that a real-valued function $\mu^\ast$ defined on the power set $\mathcal P(\mathbb N)$ 
of $\mathbb N$ 
is an \evid{arithmetic upper density} (on $\mathbb N$) if, for all $X, Y \subseteq \mathbb N$, the following conditions are satisfied:

\begin{enumerate}[label={\rm (\textsc{f}\arabic{*})}]
	\item\label{it:f1} $\mu^\ast(X) \le \mu^\ast(\mathbb N) = 1$.
	\item\label{it:f2} $\mu^\ast$ is \evid{monotone}, i.e., if $X \subseteq Y$ then $\mu^\ast(X) \le \mu^\ast(Y)$.
	\item\label{it:f3} $\mu^\ast$ is \evid{subadditive}, i.e., $\mu^\ast(X \cup Y) \le \mu^\ast(X) + \mu^\ast(Y)$.
	\item\label{it:f4} $\mu^\ast(k \cdot X + h) = \frac{1}{k} \mu^\ast(X)$ for every $k \in \mathbb N^+$ and $h \in \mathbb{N}$. 
\end{enumerate}
Moreover, we call $\mu^\ast$ an  \evid{arithmetic upper quasi-density} (on $\mathbb N$) if it satisfies \ref{it:f1}, \ref{it:f3}, and \ref{it:f4}.
\begin{remark}\label{rem:non-monotone-upper-quasi-densities}
While there do exist non-monotone arithmetic upper quasi-densities \cite[Theorem 1]{MR4054777}, such functions are not so interesting from the point of view of applications. Nevertheless, it seems meaningful to understand if monotonicity is critical to certain conclusions or can instead be dispensed with. This is our motivation for considering
arithmetic upper quasi-densities in spite of our main interest lying in the study of arithmetic upper densities (of course, the latter are a special case of the former).
\end{remark}
We let the \evid{conjugate} of an arithmetic upper quasi-density $\mu^\ast$ be the function $
\mu_\ast \colon \mathcal{P}(\mathbb N) \to \mathbb R \colon X\mapsto 1-\mu^\ast(\mathbb N\setminus X)$, and we refer to the restriction $\mu$ of $\mu^\ast$ to the set 
$$
\mathcal D := 
\{X \subseteq \mathbb N \colon \mu^\ast(X) = \mu_\ast(X)\}
$$ 
as the \evid{arithmetic quasi-density} induced by $\mu^\ast$, or simply as an arithmetic quasi-density (on $\mathbb N$) if explicit reference to $\mu^\ast$ is unnecessary. Accordingly, we call $\mathcal D$ the \evid{domain} of $\mu$ and denote it by $\text{dom}(\mu)$.

Arithmetic upper [quasi-]densities and arithmetic [quasi-]densities were introduced in \cite{MR4054777} and further studied in \cite{MR3597402,MR4360486,MR4439850}, though we are adding here the adjective ``arithmetic'' to emphasize that they assign precise values to APs (see Proposition \ref{prop:basic}\ref{it:prop:basic(4)} below). 

Notable examples of arithmetic upper densities include the upper asymptotic, upper Banach, upper analytic, upper logarithmic, upper P\'olya, and upper Buck densities, see \cite[Sect.~6 and 
Examples 4, 5, 6, and 8]{MR4054777} for details. In particular, we recall that the \evid{upper Buck density} (on $\mathbb N$) is the function
\begin{equation}\label{equ:def-upper-buck}
\mathfrak{b}^\ast \colon \mathcal{P}(\mathbb N) \to \mathbb R \colon X \mapsto \inf_{A \in \mathscr{A},\, X\subseteq A}\mathsf{d}^\ast(A),
\end{equation}
where $\mathsf d^\ast$ is the \evid{upper asymptotic density} (on $\mathbb N$), i.e., the function
\begin{equation}\label{eq:asymptotic-density}
\mathcal P(\mathbb N) \to \mathbb R \colon X \mapsto \limsup_{n \to \infty} \frac{\bigl| X \cap [1, n] \bigr|}{n}.
\end{equation}
We will write $\mathfrak b_\ast$ and $\mathfrak b$, respectively, for the conjugate of and the density induced by $\mathfrak b^\ast$. 

\begin{remark}\label{rem:asymptotic-density}
The asymptotic density $\mathsf d$ in Theorem \ref{thm:primes} is nothing but the density induced by $\mathsf d^\ast$.
\end{remark}
%
%\textcolor{red}{Ripetizione di cose gia' dette nell'intro, io leverei l'intero paragrafo e scriverei ``
We are ready to state the main theorem of the paper, whose proof we postpone to Sect.~\ref{sec:proof-of-main-thm}.
%By and large, we are going to show that Theorem \ref{thm:primes} holds with the set of primes replaced by an arbitrary non-empty set $B \subseteq \mathbb{N}$ with $\mathfrak{b}(B)=0$ and the asymptotic density replaced by an arbitrary arithmetic quasi-density $\mu$; in addition, the conclusion will be uniform in the choice of $\mu$.  More precisely, we have the following result, whose proof is postponed to Sect.~\ref{sec:proof-of-main-thm}.
%
\begin{theorem}\label{thm:mainsumsets2}
Let $B\subseteq \mathbb{N}$ be a non-empty set 
such that $\mathfrak{b}(B)=0$. 
%covering $o(n!)$ residue classes modulo $n!$ as $n\to \infty$. 
Then, for each $\alpha \in [0,1]$, there exists $A\subseteq \mathbb{N}$ such that $A \in \mathrm{dom}(\mu)$ and $\mu(A+B)=\alpha$ for all arithmetic quasi-densities $\mu$. 
\end{theorem} 

The sets $B\subseteq \mathbb{N}$ such that $\mathfrak{b}(B)=0$ have been studied in \cite{MR4360486}, where they are called ``\emph{small sets}''. 
Since $\mathfrak{b}$ is monotone and subadditive, it is clear that the family of small sets is closed under finite unions and subsets. 
%Recalling the main results in \cite{MR4360486}, 
Examples of small sets include 
the finite sets, 
the factorials, 
the perfect powers, 
and the primes. 
One may be tempted to conjecture that a set $B\subseteq \mathbb{N}$ is small if it is ``sufficiently sparse''. However, the property of being small depends on the distribution of $B$ through the APs of $\mathbb{N}$ (see Proposition  \ref{prop:basic}\ref{it:prop:basic(5)}). 
E.g., the set $\{n!+n \colon n \in \mathbb{N}\}$ is \emph{not} small (its upper Buck density is $1$), but is sparse by any standard.

To date, it is not known whether non-monotone arithmetic quasi-densities do exist (cf.~Remark \ref{rem:non-monotone-upper-quasi-densities}). 
However, arithmetic quasi-densities satisfy a weak form of monotonicity (implicit in the proof of Proposition \ref{prop:basic}) that will be enough for our goals.
%Therefore, we cannot rely on the monotonicity of an arbitrary quasi-density $\mu$ for the construction of the set $A$ in the proof of Theorem \ref{thm:mainsumsets2}.

%%%%%%%%%%%%%%%%%%%%%%%%%%%%%%%%%%%%%%%%%%%%%%
%%%%%%%%%%%%%%%%%%%%%%%%%%%%%%%%%%%%%%%%%%%%%%
%\section{Main result}\label{sec:mainresult}

%\begin{theorem}
%Fix $\alpha \in [0, 1]$ and let $B$ be a non-empty subset of $\mathbb N$ such that $\mathfrak b(B) = 0$. There then exists a set $A \in \mathrm{dom}(\mathfrak b)$ such that $A+B \in \mathrm{dom}(\mathfrak b)$ and $\mathfrak b(A + B) = \alpha$.
%\end{theorem}

%We denote by $\mathfrak b^\ast$ the \evid{upper Buck density} on $\mathbb N$ and by $\mathfrak b_\ast$ its \evid{dual}, i.e., the function $\mathcal P(\mathbb N) \to \mathbb R \colon X \mapsto 1 -  \mathfrak b^\ast(\mathbb N \setminus X)$. Accordingly, we define the \evid{Buck density} $\mathfrak b$ as the restriction of $\mathfrak b^\ast$ to the set
%\begin{equation}
%\{X \in \mathcal P(\mathbb N) \colon \mathfrak b^\ast(X) = \mathfrak b_\ast(X)\},
%\end{equation}
%herein denoted by $\mathrm{dom}(\mathfrak b)$ and referred to as the \evid{domain} (\evid{of definition}) of $\mathfrak b$. 

%%%%%%%%%%%%%%%%%%%%%%%%%%%%%%%%%%%%%%%%%%%%%%
\section{Proofs}
\label{sec:proof-of-main-thm}
%
%\textcolor{red}{If $\alpha = 0$ or} $\alpha = 1$, then the conclusion is obvious (by taking $A = \{0\}$ in the former case and $A = \mathbb N$ in the latter). So, we assume from now on that $0 < \alpha < 1$. We divide the remainder of the proof into a series of three claims.
%, in such a way that there is an integer $n_0 \ge 1$ such that $0 < \alpha - 1/n_0$ and $\alpha + 1/n_0 < 1$. 
%, where for an integer $m \ge 0$ and a set $X \subseteq \mathbb N$ we put $\Delta_m(X) := m! \cdot \mathbb N + X$.

To start with, we collect some basic properties of [upper and lower] quasi-densities that will be used, possibly without further comment, in the proof of Theorem \ref{thm:mainsumsets2}.

\begin{proposition}\label{prop:basic}
Let $\mu_\ast$ be the conjugate of an arithmetic upper quasi-density $\mu^\ast$ on $\mathbb N$, and $\mu$ be the density induced by $\mu^\ast$. Then the following hold:
	\begin{enumerate}[label={\rm (\roman{*})}]
		\item \label{it:prop:basic(1)} $\mathfrak b_\ast(X) \le \mu_\ast(X)\le \mu^\ast(X) \le \mathfrak b^\ast(X)$ for every $X\subseteq \mathbb N$.
		\item \label{it:prop:basic(2)} If $X \subseteq Y \subseteq \mathbb N$, then $\mathfrak b_\ast(X) \le \mathfrak b_\ast(Y)$.
		\item \label{it:prop:basic(3)} $\mathscr{A}\subseteq \mathrm{dom}(\mathfrak{b}) \subseteq \mathrm{dom}(\mu)$ and $\mu(X)=\mathfrak{b}(X)$ for every $X \in \mathrm{dom}(\mathfrak b)$.
		\item \label{it:prop:basic(4)} If $k \in \mathbb N^+$ and $H \subseteq \llb 0, k-1 \rrb$, then $k \cdot \mathbb N + H \in \mathrm{dom}(\mathfrak b)$ and $\mathfrak b(k \cdot \mathbb N + H) = \frac{|H|}{k}$.
		%
		%\item \label{it:prop:basic(6)} If $X\subseteq \mathbb N$ is finite, then $X \in \mathrm{dom}(\mathfrak b)$ and $\mathfrak b(X)=0$. 
		%
		%\item \label{it:prop:basic(7)} If $X \in \mathrm{dom}(\mathfrak b)$, $Y \subseteq \mathbb N$, and $\mathfrak b^\ast(Y) = 0$, then $X \cup Y \in \mathrm{dom}(\mathfrak b)$ and $\mathfrak b(X \cup Y) = \mathfrak b(X)$.
        %
        \item \label{it:prop:basic(5)} $\mathfrak{b}(X)=0$ if and only if $X$ covers $o(n!)$ residue classes modulo $n!$ as $n\to \infty$.
        \item \label{it:prop:basic(6)} If $X \in \mathscr A$, then $X + Y \in \mathscr A$ for every $Y \subseteq \mathbb N$.
	\end{enumerate}
\end{proposition}

\begin{proof}
See \cite[Proposition 2.1]{MR4439850} for items \ref{it:prop:basic(1)}--\ref{it:prop:basic(4)} and \cite[Proposition 2.6]{MR4360486} for item \ref{it:prop:basic(5)}. As for \ref{it:prop:basic(6)}, let $X \in \mathscr{A}$ and $Y \subseteq \mathbb N$. There then exist $k \in \mathbb{N}^+$ and $H\subseteq \llb 0,k-1\rrb$ such that $X=k\cdot \mathbb{N}+H$, and hence
%It then follows by the identity 
$$
X+Y=\bigcup_{y \in Y}(X+y)=\bigcup_{h \in H+Y}(k\cdot \mathbb{N}+h)=k\cdot \mathbb{N}+H',
$$
where $H'$ is the finite set $\bigcup_{i}\{\min((H+Y)\cap (k\cdot \mathbb{N}+i)\}$ and the union is extended over all $i \in \llb 0,k-1\rrb$ such that $(H+Y)\cap (k\cdot \mathbb{N}+i) \ne \emptyset$ (with the understanding that an empty union is the empty set). 
\end{proof}

We are ready for the proof of our main result. Note that the special case of a non-empty finite $B \subseteq \mathbb N$ was settled in \cite[Theorem 1.2]{MR4439850} by a different argument.

\begin{proof}
[Proof of Theorem \ref{thm:mainsumsets2}]
If $\alpha = 1$, then the conclusion is obvious (by taking $A = \mathbb N$). So, we assume from now on that $0 \le \alpha < 1$. We divide the remainder of the proof into a series of three claims. 

\vskip 0.15cm

\begin{claim}\label{claimA}
There exist a sequence $(H_n)_{n \ge 1}$ of (non-empty) subsets of $\mathbb N$ with $H_n \subseteq \llb 0, n! - 1 \rrb$ and a sequence $(h_n)_{n \ge 1}$ with $h_n \in H_n$ such that, for all $n \ge 1$, the following hold (we set $H_n' := H_n \setminus \{h_n\}$ for ease of notation):
\begin{enumerate}[label=\textup{(\roman{*})}]
\item\label{claimA(i)} $\mathfrak b^\ast(n! \cdot \mathbb N + H_n' + B) \le \alpha < \mathfrak b^\ast(n! \cdot \mathbb N + H_n +  B)$.

\item\label{claimA(ii)} $n! \cdot \mathbb N + H_n' \subseteq H_{n+1} + (n+1)! \cdot \mathbb N \subseteq n! \cdot \mathbb N + H_n$.
\end{enumerate}
\end{claim}
\begin{proof}[Proof of Claim \ref{claimA}]
\ref{claimA(i)} 
We proceed by induction. 
The base case is clear by taking $H_1:=\{0\}$ and $h_1:=0$. 

As for the inductive step, fix $m \ge 1$ and suppose we have already found a set $H_m \subseteq \llb 0, m! - 1 \rrb$ 
and an integer $h_m \in H_m$ 
such that the claimed inequality holds for $n = m$.  
Since 
\begin{equation}\label{equ:claimA(i)(1)}
H_m + m! \cdot \mathbb N = (m! \cdot \mathbb N + H_m') \cup ((m+1)! \cdot \mathbb N + h_m + \llb 0, m \rrb \cdot m!) 
\end{equation}
and $\mathfrak b^\ast(n! \cdot \mathbb N + H_m' + B) \le \alpha < \mathfrak b^\ast(n! \cdot \mathbb N + H_m + B)$,
there exists a least $k_{m+1} \in \llb 0, m \rrb$ such that 
\begin{equation}\label{equ:claimA(i)(2)}
\mathfrak b^\ast((H_m' + m! \cdot \mathbb N) \cup (h_m + \llb 0, k_{m+1} \rrb \cdot m! + (m+1)! \cdot \mathbb N)) > \alpha
\end{equation} 
Consequently, we define
\begin{equation}\label{equ:claimA(i)(3)}
H_{m+1}:=(H_m'+\llb 0,m\rrb \cdot m!) \cup (h_m + \llb 0, k_{m+1} \rrb \cdot m!)
\quad\text{and}\quad 
h_{m+1} := h_m + k_{m+1} \cdot m!.
\end{equation}
It is thus clear from Eq.~\eqref{equ:claimA(i)(2)} and the minimality of $k_{m+1}$ that the claimed inequality is also true for $n = \allowbreak m+1$. By induction, this is enough to complete the proof.

\vskip 0.15cm

\ref{claimA(ii)} 
This is now a straightforward consequence of the recursive construction of the sequences $(H_n)_{n \ge 1}$ and $(h_n)_{n \ge 1}$ as given in the inductive step of the proof of item \ref{claimA(i)}. 
\end{proof}

%%%%%%%%%%%%%%%%%%%%%%%%%%%%%%%%
\begin{claim}\label{claimB}
Set $A := \bigcap_{n \ge 1} A_n$, where $A_n := n! \cdot \mathbb N + H_n$. Then $A \in \mathrm{dom}(\mathfrak b)$. 
\end{claim}

\begin{proof}
Pick $n \in \mathbb N^+$. It follows from Claim \ref{claimA}\ref{claimA(ii)} that 
$A_n\setminus (n!\cdot \mathbb{N}+h_n) \subseteq A \subseteq A_n$. Considering that $A_n$ and $A_n\setminus (n!\cdot \mathbb{N}+h_n)$ are both in $\mathscr{A}$ and $\mathscr A$ is contained in $\mathrm{dom}(\mathfrak b)$, we obtain 
\begin{equation}\label{eq:claimB(1)}
\mathfrak{b}(A_n)-\frac{1}{n!}\le \mathfrak{b}_\ast(A)\le \mathfrak{b}^\ast(A) \le \mathfrak{b}(A_n).
\end{equation}
On the other hand, Claim \ref{claimA}\ref{claimA(ii)} gives $A_{n+1} \subseteq A_n$. Since $\mathfrak b$ is monotone, it follows that $\mathfrak b(A_n)$ tends to a limit as $n \to \infty$. So, we get from Eq.~\eqref{eq:claimB(1)} that $\mathfrak{b}_\ast(A) = \mathfrak{b}^\ast(A) = \lim_n \mathfrak b(A_n)$ and hence $A \in \mathrm{dom}(\mathfrak b)$.
\end{proof}

\begin{claim}\label{claimC}
$A+B \in \mathrm{dom}(\mathfrak b)$ and $\mathfrak b(A+B) = \alpha$.
\end{claim}

\begin{proof}
Fix $n \in \mathbb N^+$. We gather from Claim \ref{claimA}\ref{claimA(ii)} that 
$$
A_n\setminus (n!\cdot \mathbb{N}+h_n) +B\subseteq A+B \subseteq A_n+B.
$$
Considering that, by Proposition \ref{prop:basic}\ref{it:prop:basic(6)}, $X \in \mathscr{A}$ yields $X+B \in \mathscr{A}$ and $\mathscr{A}\subseteq \mathrm{dom}(\mathfrak{b})$, it follows that
\begin{equation}\label{main:ineq(1)}
\mathfrak{b}(A_n\setminus (n!\cdot \mathbb{N}+h_n) +B)
\le \mathfrak{b}_\ast(A+B)
\le \mathfrak{b}^\ast(A+B)
\le \mathfrak{b}(A_n +B).
\end{equation}
Now, fix $\varepsilon > 0$. 
Since $\mathfrak{b}(B)=0$ there exists $n_\varepsilon \in \mathbb N^+$ such that $B$ covers at most $\varepsilon \cdot n!$ residue classes modulo $n!$ for all $n\ge n_\varepsilon$. 
Consequently, we obtain from the subadditivity of $\mathfrak b^\ast$ and Claim \ref{claimA}\ref{claimA(i)} that 
\begin{equation}\label{main:ineq(2)}
\mathfrak{b}(A_{n_\varepsilon} + B)\le \mathfrak{b}(A_{n_\varepsilon} \setminus (n_\varepsilon! \cdot \mathbb{N} + h_{n_\varepsilon}) +B)+\mathfrak{b}(n_\varepsilon! \cdot \mathbb{N} + h_{n_\varepsilon} + B) \le \alpha + \varepsilon.
\end{equation}
On the other hand, the first inequality in the last display and Claim \ref{claimA}\ref{claimA(i)} imply
%On the other hand, \textcolor{red}{we have analogously}
\begin{equation}\label{main:ineq(3)}
\mathfrak{b}(A_{n_\varepsilon} \setminus ({n_\varepsilon}! \cdot \mathbb{N} + h_{n_\varepsilon}) + B)
\ge \mathfrak{b}(A_{n_\varepsilon} + B) - \mathfrak{b}({n_\varepsilon}! \cdot \mathbb{N} + h_{n_\varepsilon} + B) > \alpha - \varepsilon.
\end{equation}
Therefore, we get from Eqs.~\eqref{main:ineq(1)}--\eqref{main:ineq(3)} that $\alpha - \varepsilon < \mathfrak b_\ast(A + B) \le \mathfrak b^\ast(A + B) \le \alpha + \varepsilon$ for every $\varepsilon > 0$, which suffices to conclude that $A+B \in \mathrm{dom}(\mathfrak b)$ and $\mathfrak b(A + B) = \alpha$.
\end{proof}
By Proposition \ref{prop:basic}\ref{it:prop:basic(3)} and
Claims \ref{claimB} and \ref{claimC}, this is enough to finish the proof of the theorem.
\end{proof}

\begin{proof}[Proof of Theorem \ref{thm:primes}]
Straightforward from Theorem \ref{thm:mainsumsets2} an Remark \ref{rem:asymptotic-density}, when considering that the Buck density of the set of primes is zero (as already noted in Sect.~\ref{sec:densities}).
\end{proof}

As a final remark, we point out that, mutatis mutandis, all the results of this paper carry over 
%from arithmetic [upper] quasi-densities on $\mathbb N$
to arithmetic [upper] quasi-densities on $\mathbb Z$, in the same spirit of \cite{MR3597402, MR4054777, MR4360486, MR4439850}.

%\nocite{*}
\bibliographystyle{amsplain}
%\bibliography{biblio}

\end{document}